\documentclass[12pt, twoside]{article}
\usepackage{amsmath,amsthm,amssymb}
\usepackage{times}
\usepackage{enumerate}

\def\subjclass#1{{\renewcommand{\thefootnote}{}%
\footnote{\hspace{-0.6cm}\emph{Mathematics Subject Classification
(2020):} #1}}} 
\def\subj#1{{\renewcommand{\thefootnote}{}%
\footnote{\hspace{-0.6cm}\emph{keywords:} #1}}}

\newtheorem{thm}{Theorem}[section]

\newtheorem{lem}[thm]{Lemma}

\newtheorem{mainthm}[thm]{Main Theorem}

\theoremstyle{definition}
\newtheorem{defin}[thm]{Definition}

\newtheorem{exa}[thm]{Example}
\newtheorem{algo}[thm]{Algorithm}

\numberwithin{equation}{section}

\begin{document}


\baselineskip=17pt


\title{\bf\textsc{On the Complexity of $p$-adic continued fractions of rational number}}

\author{Rafik BELHADEF and Henri-Alex ESBELIN}

\date{}

\maketitle

 \subj{Rational number; $p$-adic
number; continued fractions}


\begin{abstract}

In this paper, we study the complexity of $p$-adic continued fractions of a rational number, which is the p-adic analogue of the theorem of Lam\'{e}. We calculate the length of Browkin expansion, and the length of Schneider expansion. Also, some numerical examples have been given.

\end{abstract}


\section{Introduction}

It is well known that, in the real case, the sequence of partial quotients of continued fraction expansion of rational number is finite. where the length of this sequence is calculated by Lame's method.  In the p-adic case, tow definitions of the continued fractions are given: the first one is the definition of Schneider [10], and the second is the definition of Ruban [9] modified by Brwokin [4]. For the expansion of rational number, Browkin [5] asked the following question: can every rational number $\alpha\in \mathbb{Q}$ be written as a finite continued fraction? If the answer is "not", determine the infinite continued fractions which correspond to rational numbers.

Based on Schneider definition, Bundschuh [6] proved that every rational number has a stationary continued fraction expansion. For Ruban definition, Laohakosol [8] and de Weger [7] proved that every rational number has a finite or periodic continued fraction expansion. Finally, Browkin based on his definition, proved in [4] that every rational number has a finite continued fraction expansion.\\
In [2] we have studied the complexity of the $p$-adic expansion of a rational number. In the present paper, we give the $p$-adic version of Lam\'{e}'s theorem, i.e. a boundary of the length of finit expansion of Browkin (see theorem 3.1), we will also give the length of the non-stationary part of the expansion of Schneider (see theorem 3.2).


\section{Definitions and properties}

To state our results, we will recall some definitions and basic facts for $p$-adic numbers and continued fractions. 

\begin{defin}
Throughout this paper $p$ is a prime number, $\mathbb{Q}$ is the field of rational numbers,
$\mathbb{Q}^*$ is the field of nonzero rational numbers and $\mathbb{R}$ is the field of real numbers. 
We use $\left|.\right|$
to denote the ordinary absolute value, $v_p$ the $p$-adic valuation,
$\left|.\right|_p$ the $p$-adic absolute value. The field of
$p$-adic numbers $\mathbb{Q}_{p}$ is the completion of $\mathbb{Q}$
with respect to the $p$-adic absolute value. Every element of $\mathbb{Q}_{p}$ can be expressed uniquely by the $p$-adic expansion $\overset{+\infty }{
\underset{n=-j}{\sum }}\alpha _{n}p^{n}$ with $\alpha _{i}\in
\{0,1,..,p-1\}$ for $i\geq -j$. In $\mathbb{Z}_{p}$ we have simply
$j=0$.

From the $p$-adic expansion of a number $x$ we can define another expansion by $x=\overset{+\infty }{\underset{n=-j}{\sum }}\alpha' _{n}p^{n}$ with $\alpha' _{i}\in
\{\dfrac{p-1}{2},..-1,0,1,.., \dfrac{p-1}{2}\}$ for $i\geq -j$ (to get this new expansion just take $\alpha' _{i}=\alpha _{i}-\dfrac{p-1}{2}$).
In this case, the $p$-adic fractional part is defined by 
\begin{equation}
\left\langle x\right\rangle _{p}=\underset{-j\leq k\leq 0}{\sum } \alpha'_{k}p^{k}=\frac{\alpha' _{-j}}{p^{j}}+\frac{\alpha' _{-j+1}}{p^{j-1}}+......+
\frac{\alpha' _{-1}}{p}+\alpha' _{0}  \label{partie fract}
\end{equation}
and the integer part is defined by
$\left[ x\right] _{p}=\underset{k=1}{\overset{+\infty }{\sum }} \alpha'_{k}p^{k}$. We have  $x=\left[ x\right] _{p}+\left\langle x\right\rangle
_{p}$.

\end{defin}


\begin{defin}\textbf{(Browkin continued fraction)} 
From a sequence $(a_i)_{i\in N}$ with values in $\mathbb{Q}^*$, we define a sequence of homographic functions of a field $\Bbbk=\mathbb{R}$ 
or $\mathbb{Q}_{p}$,  by $\left[ a,x\right]=a+\frac{1}{x}$ and $\left[a_0;a_1,a_2,...,a_n,x\right] =\left[a_0;a_1,a_2,...,a_{n-1},\left[
a_n,x\right] \right] $.
 We call $\left[a_0;a_1,a_2,...,a_{i}\right]$ the $i-th$ convergent of this
sequence.
\end{defin}

A matrix of the homographic function $\displaystyle a_k+\frac{1}{z}$
is $\left(\begin{array}{cc}a_k & 1\\ 1& 0 \\  \end{array}\right)$. Hence, a matrix of the homographic function  $[a_0;a_1,a_2, ...,
a_k,z]$ is 
$$\left(\begin{array}{cc}a_0 & 1\\ 1& 0 \\
\end{array}\right)\left(\begin{array}{cc}a_1 & 1\\ 1& 0 \\
\end{array}\right)...\left(\begin{array}{cc}a_k & 1\\ 1& 0 \\
\end{array}\right)$$
Let us denote it
$\left(\begin{array}{cc}p_k&p_{k-1}\\q_k&q_{k-1}\\
\end{array}\right)$. We have
$$ a_0 + \frac{1}{\displaystyle a_1 + \frac{1}{\displaystyle a_2 +... \frac{1}{\displaystyle a_k+\frac{y}{x}}}}=\left[a_0;a_1,a_2, ..., a_k,\frac{x}{y}\right]=\frac{p_kx +p_{k-1}y}{q_kx +q_{k-1}y}$$

\noindent and  $p_{-1}=1$, $p_0=a_0$, $p_{i+2}=a_{i+2}p_{i+1}+p_i$
and $q_{-1}=0$, $q_0=1$, $q_{i+2}=a_{i+2}q_{i+1}+q_i$  and
$$a_0 + \frac{1}{\displaystyle a_1 + \frac{1}{\displaystyle a_2 +... \frac{1}{a_k}}} = [a_0;a_1,a_2, ..., a_{k-1}, a_k] =\frac{p_k}{q_k}$$

This definition needs the $q_k$ to be nonzero.
We mention here that Browkin [4,5] considered the partial quotients $a_{1+j}$ in $
\mathbb{Z}\left[ \frac{1}{p}\right] \cap \left(
-\frac{p}{2},\frac{p}{2}\right) $ such that $\left\vert
a_0\right\vert _{p}=1$ and $\left\vert a_j\right\vert _p > 1$ for
all $j$ in $\mathbb{N}^*$.

For more general criterion of the convergence of Browkin continued fractions, see for example [3], [4] and [5].\\
In the folowing, we give an algorithm for calculating the partial quotients of the $p$-adic continued fractions of a rational number.

\begin{algo}
Let $r=\dfrac{a}{b}\in \mathbb{Q}$, given by the expansion in (2.1)
\begin{equation*}
r=\frac{u_{0}}{p^{k_{0}}}+\frac{u_{1}}{p^{k_{0}-1}}
+...+u_{k_{0}}+u_{k_{0}+1}p+u_{k_{0}+2}p^{2}+... 
\end{equation*}%
with $u _{i}\in
\{\dfrac{p-1}{2},..-1,0,1,.., \dfrac{p-1}{2}\}$, and $k_{0}=v_{p}(r)$. We can write $r$ in the form $r=\dfrac{\alpha }{\beta p^{k_{0}}}$, with $\alpha \in \mathbb{N}$ , $\beta \in \mathbb{N}^{\ast}$, and $(\alpha ,p)=1 $, $(\beta ,p)=1 $. We put
\begin{equation}
x_{0}=u_{0}+u_{1}p+u_{2}p^{2}+...+u_{k_{0}}p^{k_{0}}
\end{equation}%

Then the fractional part of $r$ is given by
$\left\langle r\right\rangle _{p}=\frac{u_{0}}{p^{k_{0}}}+\frac{u_{1}}{
p^{k_{0}-1}}+...+u_{k_{0}}=\frac{x_{0}}{p^{k_{0}}}$. \\
We put again $a_{0}=\left\langle r\right\rangle _{p}=\frac{x_{0}}{p^{k_{0}}}\in \mathbb{Z}\left[ 
\dfrac{1}{p}\right] \cap \left] -\dfrac{p}{2},\dfrac{p}{2}\right[$.
On the other hand, we have the estimate $\alpha -x_{0}\beta =0\text{mod}p^{k_{0}+1}\Longleftrightarrow \dfrac{%
\alpha -x_{0}\beta }{p^{k_{0}}}=0\text{mod}p$. We put $\dfrac{\alpha -x_{0}\beta }{p^{k_{0}}}=\beta _{1}p^{k_{1}}%
$, with  $(p, \beta _{1})=1$ , and $k_{1}=v_{p}\left( \dfrac{\alpha -x_{0}\beta 
}{p^{k_{0}}}\right) \in \mathbb{N}^{\ast }.$ \newline

Now we calculate $r_{1}$ 
\begin{equation*}
\left\{ 
\begin{array}{l}
r_{1}=\dfrac{1}{r-a_{0}}=\dfrac{1}{\dfrac{\alpha }{\beta p^{k_{0}}}-\dfrac{%
x_{0}}{p^{k_{0}}}}=\dfrac{\beta }{\beta _{1}p^{k_{1}}} \\ 
\\ 
x_{1}=\beta\beta _{1}^{-1}\text{mod}p^{1+k_{1}}
\\ 
\\ 
a_{1}=\left\langle r_{1}\right\rangle _{p}=\dfrac{x_{1}}{p^{k_{1}}}\in \mathbb{Z}\left[ 
\dfrac{1}{p}\right] \cap \left] -\dfrac{p}{2},\dfrac{p}{2}\right[ \\ 
\end{array}
\right.
\end{equation*}

Next, we search $k_{2}=v_{p}\left( \dfrac{\beta _{0}-x_{1}\beta _{1}}{
p^{k_{1}}}\right)$. Together, we calculate $r_{2}$ by the formula 
$r_{2}=\dfrac{1}{r_{1}-a_{1}}=\dfrac{\beta _{1}}{\beta _{2}p^{k_{2}}}$,  with $\beta _{2}=\dfrac{\beta _{0}-x_{1}\beta _{1}}{p^{k_{1}}p^{k_{2}}}$.\\

So, the general forms are as follows
\begin{equation*}
\left\{ 
\begin{array}{l}
\beta _{0}=\beta ,\beta _{-1}=\alpha \text{ \ \ \ \ \ and\ \ \ \ \ \ \ }
k_{n}=v_{p}\left( \dfrac{\beta _{n-2}-x_{n-1}\beta _{n-1}}{p^{k_{n-1}}}
\right)  \\ 
\\ 
\beta _{n}=\dfrac{\beta _{n-2}-x_{n-1}\beta _{n-1}}{p^{k_{n-1}}p^{k_{n}}}\text{ \ \ \ \ and \ \ \ \ } r_{n}=\dfrac{\beta _{n-1}}{\beta _{n}p^{k_{n}}}
\\ 
\\ 
x_{n}= \beta _{n-1}\beta _{n}^{-1}\text{mod}p^{1+k_{n}}\text{ \ \ \ \ and \ \ \ \ }
a_{n}=\left\langle r_{n}\right\rangle _{p}=\dfrac{x_{n}}{p^{k_{n}}}\in \mathbb{Z}\left[ \dfrac{1}{p}\right] \cap \left] -\dfrac{p}{
2},\dfrac{p}{2}\right[ 
\end{array}%
\right.
\end{equation*}
\end{algo}

According to Browkin this algorithm is finite. It means that there exists a rank $N$ such that for all $n\geq N$ we have $x_{n+1}= \beta _{n+1}= r_{n+1}=0$.\\
To prove the main theorem 3.1, we need the following lemma:

\begin{lem}
Let a linear recurrent sequence $\left( \theta _{i}\right) _{i\in\mathbb{N}}$ defined by
\begin{equation*}
\theta _{i+1}=\frac{1}{2}\theta _{i}+\frac{1}{p^{2}}\theta _{i-1},  \text{ \ \  for all \ \\ } \  i\geq 2, \text{ \ \ \ \  and \ \ \ } \ \theta
_{0}=\left\vert \beta _{0}\right\vert ,\ \theta _{1}=\left\vert \beta
_{1}\right\vert .\newline
\end{equation*}%
the following two assertions are verified:\newline
1) $\forall i\geq 2:$ \ $\left\vert \beta _{i}\right\vert \leq \theta _{i}.$
\newline
2) $\underset{i\rightarrow +\infty }{lim}\theta _{i}=0.$
\end{lem}

\begin{proof}
\textbf{1)} For all $n\geq 0$, we have
$$\left\vert x_{n}\right\vert = \left\vert u'_{0}+u_{1}p+u'_{2}p^{2}+...+u'_{k_{n}}p^{k_{n}}\right\vert \leq \dfrac{p-1}{2} (1+p+...+p^{k_{n}}) \leq \dfrac{p^{1+k_{n}}-1}{2}<\dfrac{p^{1+k_{n}}}{2}$$ then%
\begin{equation}
\left\vert \beta _{n+1}\right\vert \leq \frac{1}{p^{k_{n}}p^{k_{n+1}}}\left(
\left\vert \beta _{n-1}\right\vert +\left\vert x_{n}\right\vert \left\vert
\beta _{n}\right\vert \right) <\frac{1}{p^{2}}\left\vert \beta
_{n-1}\right\vert +\dfrac{1}{2}\left\vert \beta _{n}\right\vert \text{ }
\end{equation}
Now, we prove by induction that $\forall i\geq 2:$ \ $\left\vert \beta _{i}\right\vert \leq \theta _{i}$.\\
For $i=2$ we have $\left\vert \beta _{2}\right\vert \leq \dfrac{1}{p^{2}}\left\vert \beta
_{0}\right\vert +\dfrac{1}{2}\left\vert \beta _{1}\right\vert %
<\theta _{2}$. Suppose that $\left\vert \beta _{i}\right\vert \leq \theta _{i}$, and we show that $\left\vert \beta _{i+1}\right\vert \leq \theta _{i+1}$. Indeed, we have%
\begin{equation*}
\left\vert \beta _{i+1}\right\vert <\frac{1}{p^{2}}\left\vert \beta
_{i-1}\right\vert +\dfrac{1}{2}\left\vert \beta _{i}\right\vert <\frac{1}{%
p^{2}}\theta _{i-1}+\dfrac{1}{2}\theta _{i}=\theta _{i+1}
\end{equation*}%
then the assertion is true for $i+1$, so is true for all $\forall i\geq 2.$\newline
\textbf{2)} The characteristic polynomial of $\left( \theta
_{i}\right) _{i\in \mathbb{N}}$ is given by $2p^{2}X^{2}-p^{2}X-2=0$. this polynomial admits two real roots
\begin{equation*}
\lambda _{1}=\frac{p+\sqrt{p^{2}+16}}{4p}\text{ \ \ \ \ \ \ \ and \ \ \ \ }
\lambda _{2}=\frac{p-\sqrt{p^{2}+16}}{4p},
\end{equation*}%
therefore, the general term of the sequence is given by
\begin{equation}
\theta _{i}=\left( \frac{\left\vert \beta _{1}\right\vert -\lambda
_{2}\left\vert \beta _{0}\right\vert }{\lambda _{1}-\lambda _{2}}\right)
\lambda _{1}^{i}+\left( \frac{\left\vert \beta _{1}\right\vert +\lambda
_{1}\left\vert \beta _{0}\right\vert }{\lambda _{1}-\lambda _{2}}\right)
\lambda _{2}^{i}
\end{equation}%
so, $\underset{i\longrightarrow +\infty }{lim}\theta _{i}=0$, because $
0<\lambda _{1}<1$ and $\ \dfrac{-1}{2}<\lambda _{2}<0.$
\end{proof}


\begin{defin}\textbf{(Schneinder continued fraction)} 
From a sequence $((a_i,\alpha_i))_{i\in \mathbb{N}}$ with values in $\{1,...,p-1\}\times \mathbb{N}^*$, 
we define a sequence of homographic functions of a field $\mathbb{Q}_{p}$, for $p$ an odd prime number,   
by $\left[(a,\alpha);x\right] =a+\frac{p^\alpha}{x}$ and 
$$\left[(a_0,\alpha_0), (a_1,\alpha_1),...,(a_n,\alpha_n);x\right] =\left[(a_0,\alpha_0), ...,(a_{n-1},\alpha_{n-1});\left[ (a_n,\alpha_n);x\right] \right] $$
We call $\left[(a_0,\alpha_0), (a_1,\alpha_1),...,(a_{n-1},\alpha_{n-1});a_n\right]$ the $n-th$ convergent of this
sequence.
\end{defin}
A matrix of the homographic function $\displaystyle a_k+\frac{p^{\alpha_k}}{x}$
is $\left(\begin{array}{cc}a_k & p^{\alpha_k}\\ 1& 0 \\  \end{array}\right)$. Hence, a matrix of the homographic function  $\left[(a_0,\alpha_0), (a_1,\alpha_1),...,(a_n,\alpha_n);x\right] $ is 
$$\left(\begin{array}{cc}a_0 & p^{\alpha_0}\\ 1& 0 \\
\end{array}\right)\left(\begin{array}{cc}a_1 & p^{\alpha_1}\\ 1& 0 \\
\end{array}\right)...\left(\begin{array}{cc}a_n & p^{\alpha_n}\\ 1& 0 \\
\end{array}\right)$$
Let us denote it
$\left(\begin{array}{cc}p_n&p'_n\\q_n&q'_n\\
\end{array}\right)$. We have
$$ a_0 + \frac{p^{\alpha_0}}{\displaystyle a_1 + \frac{p^{\alpha_1}}{\displaystyle a_2 +... \frac{p^{\alpha_{n-1}}}{\displaystyle a_n+\frac{p^{\alpha_n}}{x}}}}=\left[(a_0,\alpha_0), (a_1,\alpha_1),...,(a_n,\alpha_n);x\right]=\frac{p_nx +p'_n}{q_nx +q'_n}$$

\noindent The sequences $(p_n)_{n\in \mathbb{N}}$ and  $(q_n)_{n\in \mathbb{N}}$ satisfy both the following recurrence
$$u_{n+2}=a_{n+2}u_{n+1}+u_np^{\alpha_{n+1}}$$
with $p_{-1}=1$, $p_0=a_0$ and $q_{-1}=0$, $q_0=1$.
Moreover, we have $p'_{n+1}=p_np^{\alpha_{n+1}}$  and $q'_{n+1}=q_np^{\alpha_{n+1}}$. Hence, we have

$$ a_0 + \frac{p^{\alpha_0}}{\displaystyle a_1 + \frac{p^{\alpha_1}}{\displaystyle a_2 +... \frac{p^{\alpha_{n-1}}}{ a_n}}}=\frac{p_{n-1}a_n +p'_{n-1}}{q_{n-1}a_n +q'_{n-1}}=\frac{p_n}{q_n}$$

\bigskip
For general criterion of the convergence of Schneider continued fractions see [1] and [10]. In the following, we give an algorithm for calculating the partial quotients of the Schneider continued fractions of a rational number.

\begin{algo}
\label{Algo-ratio-quot}
Let $r=$ $\dfrac{a}{b}\in \mathbb{Q}$, with $a\in \mathbb{Z}$ , $b\in \mathbb{N}^{\ast }$, and $a$, $b$, $p$ are coprime. We define a sequence $\left( y_{m}\right) _{m}$ by the following steps: \\
$\triangleright$ Put $y_{-1}=a\text{ \ , \ }y_{0}=b$.\\

$\triangleright$ Search $b_{0}\in \left\{ 0,1,2,...,p-1\right\} $
and $\alpha _{0}\in \mathbb{N}$ such that the number $y_{1}=\dfrac{a-b_{0}b}{p^{\alpha _{0}}}$
is an integer, co-prime to $y_{1}$ and $p$.\\

$\triangleright$ Search $b_{1}\in \left\{
1,2,...,p-1\right\} $ and $\alpha _{1}\in \mathbb{N}^{\ast }$ such that the number $y_{2}=\dfrac{y_{0}-b_{1}y_{1}}{p^{\alpha _{1}}}$
is an integer, co-prime to $b$ and $p$.\\

$\triangleright$ Thus, we search $b_{m}\in \left\{ 0,1,2,...,p-1\right\} $
and $\alpha _{m}\in \mathbb{N}^{\ast }$ such that the number
\begin{equation}
y_{m+1}=\dfrac{y_{m-1}-b_{m}y_{m}}{p^{\alpha _{m}}}  \label{frac2}
\end{equation}%
is an integer, co-prime to $y_{m}$ and $p$.\\

We can extract the following formula from the previous
\begin{equation*}
\frac{y_{m-1}}{y_{m}}=b_{m}+p^{\alpha _{m}}\dfrac{y_{m+1}}{y_{m}}=b_{m}+%
\dfrac{p^{\alpha _{m}}}{\dfrac{y_{m}}{y_{m+1}}}
\end{equation*}%
Hence, we have the continued fractions%
\begin{equation}
\dfrac{a}{b}=b_{0}+\dfrac{p^{\alpha _{0}}}{b_{1}+\dfrac{p^{\alpha _{1}}}{%
...b_{m}+\dfrac{p^{\alpha _{m}}}{\dfrac{y_{m}}{y_{m+1}}}}}
\end{equation}%
Using the matrix formula seen before, we obtain the following
\begin{equation}
\binom{a}{b}=\left( 
\begin{array}{cc}
b_{0} & p^{\alpha _{0}} \\ 
1 & 0%
\end{array}%
\right) \left( 
\begin{array}{cc}
b_{1} & p^{\alpha _{2}} \\ 
1 & 0%
\end{array}%
\right) ..\left( 
\begin{array}{cc}
b_{m} & p^{\alpha _{m}} \\ 
1 & 0%
\end{array}%
\right) \binom{y_{m}}{y_{m+1}}  \notag
\end{equation}%
We denote by $M_{n}$ the matrix
\begin{equation*}
\left( 
\begin{array}{cc}
b_{0} & p^{\alpha _{0}} \\ 
1 & 0%
\end{array}%
\right) \left( 
\begin{array}{cc}
b_{1} & p^{\alpha _{2}} \\ 
1 & 0%
\end{array}%
\right) ..\left( 
\begin{array}{cc}
b_{m} & p^{\alpha _{m}} \\ 
1 & 0%
\end{array}%
\right) =\left( 
\begin{array}{cc}
U_{m} & V_{m} \\ 
W_{m} & Z_{m}%
\end{array}%
\right)
\end{equation*}%
Which implies the following recurrence relation%
\begin{equation}
M_{m}=M_{m-1}\left( 
\begin{array}{cc}
b_{m} & p^{\alpha _{m}} \\ 
1 & 0%
\end{array}%
\right) =\left( 
\begin{array}{cc}
U_{m-1} & V_{m-1} \\ 
W_{m-1} & Z_{m-1}%
\end{array}%
\right) \left( 
\begin{array}{cc}
b_{m} & p^{\alpha _{m}} \\ 
1 & 0%
\end{array}%
\right)
\end{equation}%
with the initial conditions%
\begin{equation*}
M_{-1}=\left( 
\begin{array}{cc}
1 & 0 \\ 
0 & 1%
\end{array}%
\right) \text{ \ \ \ , \ \ \ \ \ }M_{0}=\left( 
\begin{array}{cc}
b_{0} & p^{\alpha _{0}} \\ 
1 & 0%
\end{array}%
\right)
\end{equation*}%
so, we have the  system
\begin{equation*}
\left\{ 
\begin{array}{c}
U_{m}=b_{m}U_{m-1}+V_{m-1}\text{ \ \ \ , \ \ \ \ }V_{m}=p^{\alpha
_{m}}U_{m-1} \\ 
\\ 
W_{m}=b_{m}W_{m-1}+Z_{m-1}\text{ \ \ \ , \ \ \ }Z_{m}=p^{\alpha _{m}}W_{m-1}%
\end{array}%
\right.
\end{equation*}%
i.e%
\begin{equation}
\left\{ 
\begin{array}{c}
U_{m}=b_{m}U_{m-1}+p^{\alpha _{m-1}}U_{m-2} \\ 
\\ 
W_{m}=b_{m}W_{m-1}+p^{\alpha _{m-1}}W_{m-2}%
\end{array}%
\right.  \label{Vecteur matric}
\end{equation}
\end{algo}


\section{Statements }
For the Brwokin continued fractions we have the following main theorem:

\begin{mainthm} \label{Brow-ratio}
Let $r=\dfrac{a}{b}\in \mathbb{Q}$, with $k_{0}=v_{p}(r)$.
Then, the set
\begin{equation*}
\left\{ n\in \mathbb{N}\diagup a_{n}=\left\langle r_{n}\right\rangle _{p}=\dfrac{x_{n}}{p^{k_{n}}}%
\right\}
\end{equation*}%
is finite, and its cardinal is increased by the real integer part
of the number
\begin{equation}
\dfrac{-1}{\log \lambda _{1}}\log \left( \dfrac{%
2\left\vert \beta _{1}\right\vert }{\lambda _{1}-\lambda _{2}}+\left\vert
\beta _{0}\right\vert \right)
\end{equation}%
with $\lambda _{1}$ , $\lambda _{2}$\ are the roots of the equation $%
2p^{2}X^{2}-p^{2}X-2=0.$\\
Where $x_{n}$, $a_{n}, r_{n}, k_{n}$ and $\beta _{n}$  are defined as in algorithm 2.3.
\end{mainthm}%

\bigskip

For the Schneider continued fractions, Bundschuh [6] proved that every rational number has a stationnary expansion, i.e.%
\begin{equation*}
\exists k\geq 0,\forall j>k:\ \alpha_{j}=1\text{ \ \ \ et \ \ \ }a_{j}=p-1.
\end{equation*}%
In the following main theorem, we will calculate the value of $k$:

\begin{mainthm} \label{Schn-ratio}
For $p\geq 3,$ let $r=\dfrac{a}{b}\in \mathbb{Q}$, given by its Schneider continued fractions
$$ a_0 + \frac{p^{\alpha_0}}{\displaystyle a_1 + \frac{p^{\alpha_1}}{\displaystyle a_2 +... \frac{p^{\alpha_{k}}}{\displaystyle a_k+\frac{p}{{\displaystyle p-1+\frac{p}{{{\displaystyle p-1+\frac{p}{...}}}}}}}}}$$
We assume that $(a_{i},\alpha _{i})=(\lambda,\alpha)\neq(p-1,1)$, for $ 0\leq i\leq k$.\\  
Then, the length of the non-stationary part is equal to 
\begin{equation*}
k=\left[ \dfrac{\ln \left\vert \theta \right\vert }{\ln \left\vert \dfrac{%
T_{2}}{T_{1}}\right\vert }\right] +1
\end{equation*}%
with $\theta =\frac{\left( T_{1}-p^{\alpha} \right) \left( a-b T_{1} \right) }{\left( T_{2}-p^{\alpha} \right) \left( a-b T_{2} \right) }$,
where $T_{1}$, $T_{2}$ \ are the roots of the equation $T^{2}-\lambda T-p^{\alpha}=0$. 

\end{mainthm}


\section{Proof of the main theorems}

\begin{proof}(Main theorem 3.1)\\
From lemma 2.4, we have $\underset{i\longrightarrow +\infty}{lim}\theta _{i}=0$. Therefore
\begin{equation}
\exists N\in \mathbb{N}:i>N\Longrightarrow \left\vert \beta _{i}\right\vert <\theta _{i}<1
\end{equation}
But $\left( \left\vert \beta _{i}\right\vert \right) _{i}$ is a sequence of positive integer numbers, so $\left\vert \beta _{i}\right\vert =0$, for all $i>N$. This shows that the algorithm for calculating $\left( a_{n}\right)
_{n\in \mathbb{N}}$ stops after the rank $N$; hence the set $\left\{
n\in \mathbb{N}\diagup a_{n}=\left\langle r_{n}\right\rangle _{p}=\dfrac{x_{n}}{p^{k_{n}}}%
\right\} $ is finite. Now, we prove that
\begin{equation*}
card\left\{ n\in \mathbb{N}\diagup a_{n}=\left\langle r_{n}\right\rangle _{p}=\dfrac{x_{n}}{p^{k_{n}}}%
\right\} \leq \left[ \dfrac{-\log \left( \dfrac{2\left\vert
\beta _{1}\right\vert }{\lambda _{1}-\lambda _{2}}+\left\vert \beta
_{0}\right\vert \right) }{\log \lambda _{1}}\right]
\end{equation*}
Indeed, from (4.1) we have for all $i>N$, $\theta _{i}<1$. On the other hand we have $\left\vert \lambda _{2}\right\vert <\lambda _{1}$ then

\begin{eqnarray*}
\left\vert \theta _{i}\right\vert &=&\left\vert \left( \frac{\left\vert
\beta _{1}\right\vert -\lambda _{2}\left\vert \beta _{0}\right\vert }{%
\lambda _{1}-\lambda _{2}}\right) \lambda _{1}^{i}+\left( \frac{\lambda
_{1}\left\vert \beta _{0}\right\vert +\left\vert \beta _{1}\right\vert }{%
\lambda _{1}-\lambda _{2}}\right) \lambda _{2}^{i}\right\vert \\
&<&\lambda _{1}^{i}\left( \frac{2\left\vert \beta _{1}\right\vert }{\lambda
_{1}-\lambda _{2}}+\left\vert \beta _{0}\right\vert \right)
\end{eqnarray*}
In order to have that $\theta _{N+1}<1$, it is sufficient to get $$\lambda _{1}^{N+1}\left( \frac{2\left\vert \beta _{1}\right\vert }{
\lambda _{1}-\lambda _{2}}+\left\vert \beta _{0}\right\vert \right) \leq 1$$
Furthermore
\begin{equation}
N+1\geq \frac{-\log \left( \dfrac{2\left\vert \beta
_{1}\right\vert }{\lambda _{1}-\lambda _{2}}+\left\vert \beta
_{0}\right\vert \right) }{\log \lambda _{1}}
\end{equation}%
So, we take
$$N=\left[ \dfrac{-\log \left( \dfrac{2\left\vert \beta
_{1}\right\vert }{\lambda _{1}-\lambda _{2}}+\left\vert \beta
_{0}\right\vert \right) }{\log \lambda _{1}}\right]$$
\end{proof}

\begin{exa}
For $p=3$, and $\frac{a}{b}=\frac{77}{18}$, we have $\beta _{-1}=\alpha =77$, $\beta _{0}=2$, $k_{0}=2$. The table 1 comes from algorithm 2.3

$$\textbf{Table 1: The $3$-adic Browkin continued fraction of $\frac{77}{18}$}$$
 \begin{equation*}
\begin{tabular}{|l|c|c|c|}
\hline
$n$ & $0$ & $1$ & $2$ \\ \hline
& \multicolumn{1}{|l|}{%
\begin{tabular}{lll}
$k_{0}$ & $\beta _{0}$ & $a_{0}$ \\ 
2 & 2 & $\frac{25}{9}$%
\end{tabular}%
} & \multicolumn{1}{|l|}{%
\begin{tabular}{lll}
$k_{1}$ & $\beta _{1}$ & $a_{1}$ \\ 
1 & 1 & $\frac{11}{3}$%
\end{tabular}%
} & \multicolumn{1}{|l|}{%
\begin{tabular}{lll}
$k_{2}$ & $\beta _{2}$ & $a_{2}$ \\ 
1 & -1 & $-\frac{1}{3}$%
\end{tabular}%
} \\ \hline
\end{tabular}
\end{equation*}
\bigskip
From theorem 3.1, we have $\lambda _{1}=\frac{2}{3}$, $\lambda _{2}=-\frac{1}{
6}$ then $k=3$. So, the Browkin continued fractions of $\frac{77}{18}$ are given by
$$\dfrac{77}{18}=\frac{25}{9}+\frac{1}{\dfrac{11}{3}+\dfrac{3}{-\dfrac{1}{3}}}$$
\end{exa}

\begin{exa}
For $p=3$, and $\frac{a}{b}=\frac{365}{54}$, we put $\beta _{-1}=\alpha
=365,$ $\beta _{0}=2,$ $k_{0}=3$. The table 2 comes from the algorithm 2.3

$$\textbf{Table 2: The $3$-adic Browkin continued fraction of $\frac{365}{54}$}$$
 \begin{equation*}
\begin{tabular}{|l|c|c|c|c|c|}
\hline
$n$ & $0$ & $1$ & $2$ & $3$ & $4$ \\ \hline
& \multicolumn{1}{|l|}{%
\begin{tabular}{lll}
$k_{0}$ & $\beta _{0}$ & $a_{0}$ \\ 
3 & 2 & $-\frac{20}{27}$%
\end{tabular}%
} & \multicolumn{1}{|l|}{%
\begin{tabular}{lll}
$k_{1}$ & $\beta _{1}$ & $a_{1}$ \\ 
1 & 5 & $\frac{4}{3}$%
\end{tabular}%
} & \multicolumn{1}{|l|}{%
\begin{tabular}{lll}
$k_{2}$ & $\beta _{2}$ & $a_{2}$ \\ 
1 & -2 & $\frac{2}{3}$%
\end{tabular}%
} & \multicolumn{1}{|l|}{%
\begin{tabular}{lll}
$k_{3}$ & $\beta _{3}$ & $a_{3}$ \\ 
1 & 1 & $\frac{7}{3}$%
\end{tabular}%
} & \multicolumn{1}{|l|}{%
\begin{tabular}{lll}
$k_{4}$ & $\beta _{4}$ & $a_{4}$ \\ 
1 & -1 & $-\frac{1}{3}$%
\end{tabular}%
} \\ \hline
\end{tabular}
\end{equation*}
\bigskip

From theorem 3.1 we have $\lambda _{1}=\frac{2}{3}$, $\lambda _{2}=-\frac{1}{6}$ then $k=6$.
So, the Browkin continued fractions of $\frac{365}{54}$ are given by

$$\dfrac{365}{54}=-\frac{20}{27}+\frac{1}{\dfrac{4}{3}+\dfrac{1}{\dfrac{2}{3}+%
\dfrac{1}{\dfrac{7}{3}+\dfrac{1}{-\dfrac{1}{3}}}}}$$

\end{exa}

\begin{exa}
For $p=5$, the expansion of $\frac{a}{b}=-\frac{1793}{100}$ with
the coefficients in $\left\{ -2,-1,0,+1,+2\right\} $ are given by
$$-\dfrac{1793}{100}=-\frac{2}{5^{2}}+\frac{2}{3}%
-2-2.5+1.5^{2}+1.5^{3}+1.5^{4}+...$$
we put $\beta _{-1}=\alpha =-1793,$ $\beta _{0}=4,$ $k_{0}=2$. The table 3 comes from the algorithm 2.3

$$\textbf{Table 3: The $5$-adic Browkin continued fraction of $\frac{1793}{100}$}$$
 \begin{equation*}
\begin{tabular}{|l|c|c|c|c|}
\hline
$n$ & $0$ & $1$ & $2$ & $3$ \\ \hline
& \multicolumn{1}{|l|}{%
\begin{tabular}{lll}
$k_{0}$ & $\beta _{0}$ & $a_{0}$ \\ 
2 & 4 & $-\frac{42}{25}$%
\end{tabular}%
} & \multicolumn{1}{|l|}{%
\begin{tabular}{lll}
$k_{1}$ & $\beta _{1}$ & $a_{1}$ \\ 
1 & 13 & $-\frac{8}{5}$%
\end{tabular}%
} & \multicolumn{1}{|l|}{%
\begin{tabular}{lll}
$k_{2}$ & $\beta _{2}$ & $a_{2}$ \\ 
1 & -4 & $-\frac{3}{5}$%
\end{tabular}%
} & \multicolumn{1}{|l|}{%
\begin{tabular}{lll}
$k_{3}$ & $\beta _{3}$ & $a_{3}$ \\ 
1 & 1 & $-\frac{4}{5}$%
\end{tabular}%
} \\ \hline
\end{tabular}
 \end{equation*}

From theorem 3.1 we have $\lambda _{1}=\frac{5+\sqrt{41}}{20}$, $\lambda
_{2}=\frac{5-\sqrt{41}}{20}$ then $k=6$.

So, the Browkin continued fractions of $-\frac{1793}{100}$ are given by

\[
-\dfrac{1793}{100}=-\dfrac{42}{25}+\frac{1}{-\dfrac{8}{5}+\dfrac{1}{-\dfrac{3%
}{5}+\dfrac{1}{-\dfrac{4}{5}}}}
\]

\end{exa}

\bigskip


\begin{proof} (Main theorem 3.2)

In the stationnay case of Schneider continued fractions of rational number, we have $\forall m\geq k:$ $\left\vert y_{m+1}\right\vert =1$, $%
b_{m+1}=p-1$, $\alpha _{m+1}=1$. If we have $(a_{i},\alpha _{i})=(p-1,1)$, for $ 0\leq i\leq k$, then $\dfrac{a}{b}=-1$.\\
In the following, we suppose $(a_{i},\alpha _{i})=(\lambda,\alpha)\neq(p-1,1)$, for $ 0\leq i\leq k$, the matrix formula becomes%
\begin{eqnarray*}
\binom{a}{b} &=&\left( 
\begin{array}{cc}
\lambda & p^{\alpha} \\ 
1 & 0%
\end{array}%
\right) \left( 
\begin{array}{cc}
\lambda & p^{\alpha} \\ 
1 & 0%
\end{array}%
\right) ..\left( 
\begin{array}{cc}
\lambda & p^{\alpha} \\ 
1 & 0%
\end{array}%
\right) \binom{+ 1}{- 1} \\
&=& \left( \begin{array}{cc}
U_{k} & V_{k} \\ 
W_{k} & Z_{k}%
\end{array}%
\right) \binom{+ 1}{- 1}
\end{eqnarray*}%
So, we find
\begin{equation}
\left\{ 
\begin{array}{l}
a=U_{k}-V_{k}=U_{k}-p^{\alpha}U_{k-1} \\ 
\\
b=W_{k}-Z_{k}=W_{k}-p^{\alpha}W_{k-1}%
\end{array}%
\right.  \label{a et b}
\end{equation}%
Thus
\begin{equation*}
\left\{ 
\begin{array}{l}
U_{n}=\lambda U_{n-1}+p^{\alpha}U_{n-2} \\
\\ 
W_{n}=\lambda W_{n-1}+p^{\alpha}W_{n-2}%
\end{array}%
\right.
\end{equation*}
which means $\left( U_{n}\right) _{n}$ and $\left( W_{n}\right) _{n}$ are two linear recurrence sequences. Their general forms are as follows%
\begin{equation*}
\left\{ 
\begin{array}{l}
U_{n}= \dfrac{1}{\sqrt{4p^{\alpha}+\lambda^2}} \left( T_{2}^{n+1}-T_{1}^{n+1}\right) \\ 
\\ 
W_{n}= \dfrac{1}{\sqrt{4p^{\alpha}+\lambda^2}} \left( T_{2}^{n}-T_{1}^{n}\right)%
\end{array}
\right.
\end{equation*}
with the first terms $\ U_{0}=1$, $U_{1}=\lambda$, $W_{0}=0$, $W_{1}=1$. Moreover, $%
T_{2}$ and $T_{1}$ are the roots of the characteristic equation: \ $
T^{2}-\lambda T-p^{\alpha}=0$, they are given by%
\begin{equation}
\left\{ 
\begin{array}{c}
T_{1}=\dfrac{\lambda-\sqrt{4p^{\alpha}+\lambda^2}}{2} \\ 
\\ 
T_{2}=\dfrac{\lambda+\sqrt{4p^{\alpha}+\lambda^2}}{2}%
\end{array}%
\right.
\end{equation}%
Back to the formula (\ref{a et b}), we obtain%
\begin{equation*}
\left\{ 
\begin{array}{l}
a=\dfrac{1}{\sqrt{4p^{\alpha}+\lambda^2}} \left(
T_{2}^{k}\left( T_{2}-p^{\alpha}\right) -T_{1}^{k}\left( T_{1}-p^{\alpha}\right) \right) \\ 
\\ 
b=\dfrac{1}{\sqrt{4p^{\alpha}+\lambda^2}} \left(
T_{2}^{k-1}\left( T_{2}-p^{\alpha}\right) -T_{1}^{k-1}\left( T_{1}-p^{\alpha}\right) \right)%
\end{array}%
\right.
\end{equation*}%
We divide the two numbers%
\begin{equation*}
\frac{a}{b}=\frac{ 
T_{2}^{k}\left( T_{2}-p^{\alpha}\right) -T_{1}^{k}\left( T_{1}-p^{\alpha}\right) 
}{ T_{2}^{k-1}\left( T_{2}-p^{\alpha}\right) -T_{1}^{k-1}\left( T_{1}-p^{\alpha}\right)
}
\end{equation*}%
then%

\begin{equation}
\frac{T_{2}^{k-1}}{T_{1}^{k-1}}=\frac{\left( T_{1}-p^{\alpha} \right) \left( a-b T_{1} \right) }{\left( T_{2}-p^{\alpha} \right) \left( a-b T_{2} \right) }
\end{equation}%
hence%
\begin{equation*}
\left( \left\vert \frac{T_{2}}{T_{1}}\right\vert \right) ^{k-1}=\left\vert 
\frac{\left( T_{1}-p^{\alpha} \right) \left( a-b T_{1} \right) }{\left( T_{2}-p^{\alpha} \right) \left( a-b T_{2} \right) }\right\vert
\end{equation*}%
finally%
\begin{equation}
k=\left[ \dfrac{\ln \left\vert \theta \right\vert }{\ln \left\vert \frac{%
T_{2}}{T_{1}}\right\vert }\right] +1
\end{equation}%
with%
$$\theta =\frac{\left( T_{1}-p^{\alpha} \right) \left( a-b T_{1} \right) }{\left( T_{2}-p^{\alpha} \right) \left( a-b T_{2} \right) }$$
it's clear that $\left\vert T_{2}\right\vert >\left\vert T_{1}\right\vert $ and $T_{1}\neq 0$, so it is necessary that $\left\vert \theta \right\vert >1$ in order that $k$ either well defined. 
However, we have the following cases:\newline
*) If $b> 1$ and $a>\ b T_{2} $ or $b T_{1} <\ a <\ b T_{2} $, then $\left\vert \theta \right\vert >1$. \newline
\newline
*) If $b<-1$ and $a<\ b T_{2} $ or $b T_{2} <\ a <\ b T_{1} $, then $\left\vert \theta \right\vert >1$.

\end{proof}


\begin{exa}

For $p=3$ and $\frac{a}{b}=\frac{2}{5}$, we have $\lambda =1$ and $\alpha
=1$. We put $y_{-1}=2,$ $y_{0}=5$. The table 4 comes from the algorithm 2.6

\newpage

$$\textbf{Table 4: The $3$-adic Schneider continued fraction of $\frac{2}{5}$}$$
\begin{equation*}
\begin{tabular}{|l|c|c|c|c|c|}
\hline
$n$ & $0$ & $1$ & $2$ & $3$ & $4$ \\ \hline
& \multicolumn{1}{|l|}{
\begin{tabular}{lll}
$b_{0}$ & $a_{0}$ & $y_{1}$ \\ 
1 & 1 & -1%
\end{tabular}%
} & \multicolumn{1}{|l|}{
\begin{tabular}{lll}
$b_{1}$ & $a_{1}$ & $y_{2}$ \\ 
1 & 1 & 2%
\end{tabular}%
} & \multicolumn{1}{|l|}{
\begin{tabular}{lll}
$b_{2}$ & $a_{2}$ & $y_{3}$ \\ 
1 & 1 & -1%
\end{tabular}%
} & \multicolumn{1}{|l|}{
\begin{tabular}{lll}
$b_{3}$ & $a_{3}$ & $y_{4}$ \\ 
1 & 1 & 1%
\end{tabular}%
} & \multicolumn{1}{|l|}{
\begin{tabular}{lll}
$b_{4}$ & $a_{4}$ & $y_{5}$ \\ 
2 & 1 & -1%
\end{tabular}%
} \\ \hline
$n$ & $5$ & $6$ & $7$ & $8$ & $9$ \\ \hline
& \multicolumn{1}{|l|}{%
\begin{tabular}{lll}
$b_{5}$ & $a_{5}$ & $y_{6}$ \\ 
2 & 1 & -1%
\end{tabular}%
} & \multicolumn{1}{|l|}{%
\begin{tabular}{lll}
$b_{6}$ & $a_{6}$ & $y_{7}$ \\ 
2 & 1 & -1%
\end{tabular}%
} & \multicolumn{1}{|l|}{%
\begin{tabular}{lll}
$b_{7}$ & $a_{7}$ & $y_{8}$ \\ 
2 & 1 & -1%
\end{tabular}%
} & \multicolumn{1}{|l|}{%
\begin{tabular}{lll}
$b_{8}$ & $a_{8}$ & $y_{9}$ \\ 
2 & 1 & -1%
\end{tabular}%
} & \multicolumn{1}{|l|}{%
\begin{tabular}{lll}
$b_{9}$ & $a_{9}$ & $y_{10}$ \\ 
2 & 1 & -1%
\end{tabular}%
} \\ \hline
\end{tabular}%
\end{equation*}

\bigskip
From theorem 3.2 we have $T_{1}=-1.303$, $T_{2}=2.303$. Then, $\theta =-5.523$ and $k=3$.
So the Schneider continued fractions of $\frac{2}{5}$ are given by
$$\frac{2}{5}=1+\frac{3}{1+\dfrac{3}{1+\dfrac{3}{1+\dfrac{3}{2+\dfrac{3}{2+%
\dfrac{3}{2+\dfrac{3}{2+...}}}}}}} $$
\end{exa}

\begin{exa}

For $p=3$, and $\frac{a}{b}=\frac{1259}{701}$ we have $\lambda =1$ and $%
\alpha =2$, we put $y_{-1}=1259,$ $y_{0}=701$. The table 5 comes from the algorithm 2.6
$$\textbf{Table 5: The $3$-adic Schneider continued fraction of $\frac{1259}{701}$}$$
\begin{equation*}
\begin{tabular}{|l|c|c|c|c|c|}
\hline
\multicolumn{1}{|l|}{$n$} & $0$ & $1$ & $2$ & $3$ & $4$ \\ \hline
\multicolumn{1}{|l|}{} & \multicolumn{1}{|l|}{%
\begin{tabular}{lll}
$b_{0}$ & $a_{0}$ & $y_{1}$ \\ 
1 & 2 & 62%
\end{tabular}%
} & \multicolumn{1}{|l|}{%
\begin{tabular}{lll}
$b_{1}$ & $a_{1}$ & $y_{2}$ \\ 
1 & 2 & 71%
\end{tabular}%
} & \multicolumn{1}{|l|}{%
\begin{tabular}{lll}
$b_{2}$ & $a_{2}$ & $y_{3}$ \\ 
1 & 2 & -1%
\end{tabular}%
} & \multicolumn{1}{|l|}{%
\begin{tabular}{lll}
$b_{3}$ & $a_{3}$ & $y_{4}$ \\ 
1 & 2 & 8%
\end{tabular}%
} & \multicolumn{1}{|l|}{%
\begin{tabular}{lll}
$b_{4}$ & $a_{4}$ & $y_{5}$ \\ 
1 & 2 & -1%
\end{tabular}%
} \\ \hline
\multicolumn{1}{|l|}{$n$} & $5$ & $6$ & $7$ & $8$ & $9$ \\ \hline
\multicolumn{1}{|l|}{} & \multicolumn{1}{|l|}{%
\begin{tabular}{lll}
$b_{5}$ & $a_{5}$ & $y_{6}$ \\ 
1 & 2 & 1%
\end{tabular}%
} & \multicolumn{1}{|l|}{%
\begin{tabular}{lll}
$b_{6}$ & $a_{6}$ & $y_{7}$ \\ 
2 & 1 & -1%
\end{tabular}%
} & \multicolumn{1}{|l|}{%
\begin{tabular}{lll}
$b_{7}$ & $a_{7}$ & $y_{8}$ \\ 
2 & 1 & -1%
\end{tabular}%
} & \multicolumn{1}{|l|}{%
\begin{tabular}{lll}
$b_{8}$ & $a_{8}$ & $y_{9}$ \\ 
2 & 1 & -1%
\end{tabular}%
} & \multicolumn{1}{|l|}{%
\begin{tabular}{lll}
$b_{9}$ & $a_{9}$ & $y_{10}$ \\ 
2 & 1 & -1%
\end{tabular}%
} \\ \hline
\end{tabular}%
\end{equation*}

\bigskip
From theorem 3.2 we have $T_{1}=-1.303$, $T_{2}=2.303$. Then $\theta =-73.736$ and $k=5$. So, the Schneider continued fractions of $\frac{1259}{701}$ are given by

$$
\dfrac{1259}{701}=1+\frac{9}{1+\dfrac{9}{1+\dfrac{9}{1+\dfrac{9}{1+\dfrac{9}{%
1+\dfrac{3}{2+\dfrac{3}{2+\dfrac{3}{2+\dfrac{3}{2+...}}}}}}}}} 
$$
\end{exa}

\begin{exa}

For $p=5$, and $\frac{a}{b}=\frac{3044}{673}$ we have $\lambda =3$ and $%
\alpha =2$, we put $y_{-1}=3044,$ $y_{0}=673$. The table 6 comes from the algorithm 2.6
$$\textbf{Table 6: The $3$-adic Schneider continued fraction of $\frac{3044}{673}$}$$
\begin{equation*}
\begin{tabular}{|l|c|c|c|c|c|}
\hline
\multicolumn{1}{|l|}{$n$} & $0$ & $1$ & $2$ & $3$ & $4$ \\ \hline
\multicolumn{1}{|l|}{} & \multicolumn{1}{|l|}{%
\begin{tabular}{lll}
$b_{0}$ & $a_{0}$ & $y_{1}$ \\ 
3 & 2 & 41%
\end{tabular}%
} & \multicolumn{1}{|l|}{%
\begin{tabular}{lll}
$b_{1}$ & $a_{1}$ & $y_{2}$ \\ 
3 & 2 & 22%
\end{tabular}%
} & \multicolumn{1}{|l|}{%
\begin{tabular}{lll}
$b_{2}$ & $a_{2}$ & $y_{3}$ \\ 
3 & 2 & -1%
\end{tabular}%
} & \multicolumn{1}{|l|}{%
\begin{tabular}{lll}
$b_{3}$ & $a_{3}$ & $y_{4}$ \\ 
3 & 2 & 1%
\end{tabular}%
} & \multicolumn{1}{|l|}{%
\begin{tabular}{lll}
$b_{4}$ & $a_{4}$ & $y_{5}$ \\ 
4 & 1 & -1%
\end{tabular}
} \\ \hline
\multicolumn{1}{|l|}{$n$} & $5$ & $6$ & $7$ & $8$ & $9$ \\ \hline
\multicolumn{1}{|l|}{} & \multicolumn{1}{|l|}{%
\begin{tabular}{lll}
$b_{5}$ & $a_{5}$ & $y_{6}$ \\ 
4 & 1 & -1%
\end{tabular}%
} & \multicolumn{1}{|l|}{%
\begin{tabular}{lll}
$b_{6}$ & $a_{6}$ & $y_{7}$ \\ 
4 & 1 & -1%
\end{tabular}%
} & \multicolumn{1}{|l|}{%
\begin{tabular}{lll}
$b_{7}$ & $a_{7}$ & $y_{8}$ \\ 
4 & 1 & -1%
\end{tabular}%
} & \multicolumn{1}{|l|}{%
\begin{tabular}{lll}
$b_{8}$ & $a_{8}$ & $y_{9}$ \\ 
4 & 1 & -1%
\end{tabular}%
} & \multicolumn{1}{|l|}{%
\begin{tabular}{lll}
$b_{9}$ & $a_{9}$ & $y_{10}$ \\ 
4 & 1 & -1%
\end{tabular}%
} \\ \hline
\end{tabular}%
\end{equation*}
\bigskip

From theorem 3.2 we have $T_{1}=-1.791$, $T_{2}=2.791$. Then $\theta =11.211$ and $k=5$. So, the Schneider continued fractions of $\frac{3044}{673}$ are given by

$$
\dfrac{3044}{673}=3+\frac{25}{3+\dfrac{25}{3+\dfrac{25}{3+\dfrac{25}{4+%
\dfrac{5}{4+\dfrac{5}{4+\dfrac{5}{4+\dfrac{5}{4+\dfrac{5}{4+...}}}}}}}}} 
$$
\end{exa}


Rafik Belhadef\\
LMPA, Jijel University, BP 98, Jijel, Algeria\\
Corresponding author. Email: belhadefrafik@univ-jijel.dz, rbelhadef@gmail.com
Henri-Alex Esbelin\\
LIMOS, Clermont Auvergne University, Aubi\`ere; France
\label{lastpage}
\end{document}